\documentclass[12pt]{amsart}
\usepackage{amssymb,amsmath,amsthm,slashed}
\newtheorem{theorem}{Theorem}
\newtheorem{lemma}[theorem]{Lemma}
\newtheorem{corollary}[theorem]{Corollary}
\newtheorem{conjecture}{\bf Conjecture}
\newtheorem{proposition}[theorem]{Proposition}

\theoremstyle{remark}

\newtheorem*{remark}{Remark}

\numberwithin{theorem}{section} \numberwithin{equation}{section}

\newcommand{\re}{\textnormal{Re}}

\begin{document}

\title[Donaldson invariants of $\mathbb{C}\mathrm{P}^1 \times \mathbb{C}\mathrm{P}^1$ and mock Theta Functions]
{Donaldson invariants of $\mathbb{C}\mathrm{P}^1 \times \mathbb{C}\mathrm{P}^1$ and mock Theta Functions}

\author{Andreas Malmendier}

\address{Department of Mathematics, Colby College, 
Waterville, Maine, ME 04901}
\email{andreas.malmendier@colby.edu}

\begin{abstract}
We compute the Moore-Witten regularized $u$-plane integral 
on $\mathbb{C}\mathrm{P}^1 \times \mathbb{C}\mathrm{P}^1$ directly in
a chamber where the elliptic unfolding technique fails to work. This allows us to 
determine explicit formulas for the $\mathrm{SU}(2)$ and $\mathrm{SO}(3)$-Donaldson invariants
of $\mathbb{C}\mathrm{P}^1 \times \mathbb{C}\mathrm{P}^1$ in terms of mock modular forms.
\end{abstract}

\maketitle

\section{Introduction and Statement of Results}

There are two families of Donaldson invariants of a smooth, compact, oriented, simply connected Riemannian four-manifold without boundary corresponding to the $\mathrm{SU}(2)$-gauge  theory and the $\mathrm{SO}(3)$-gauge theory with non-trivial Stiefel-Whitney class.  From the viewpoint of theoretical physics \cite{Witten2, Witten3}, these
two families of Donaldson invariants and their related Seiberg-Witten invariants of a manifold are the correlation functions of a supersymmetric topological gauge theory for the gauge group $\mathrm{SU}(2)$ and $\mathrm{SO}(3)$ respectively whose space-time is the given manifold. However, the computation
of these correlation functions is in most cases too complicated to be carried out explicitly.

Using physical considerations, Witten \cite{Witten1} argued that one should be able to compute the correlation functions in a low-energy effective field theory instead. The effective theory has the advantage of being an Abelian supersymmetric topological gauge theory. This means that the data required to define the theory only involves line bundles on the manifold. Seiberg and Witten \cite{SeibergWitten1, SeibergWitten2} argued further
that the moduli of the low-energy effective field theory are parametrized by the modular elliptic surface over 
$\Gamma_0(4) \backslash \mathbb{H}$, henceforth called the $u$-plane.

Based on this effective low-energy description of the quantum theory, 
Witten obtained an explicit formula for the generating function of the Donaldson invariants in terms of the Seiberg-Witten invariants if the manifold has Betti number $b_2^+>1$ and is of simple type. 
The simple type condition is a condition on the structural relation between the Donaldson invariants of different degree.
However, it is conjectured to be equivalent to the condition $b_2^+>1$.
Kronheimer and Mrowka developed a structure theory for the instanton invariants of manifolds of simple type \cite{KronheimerMrowka}.
They found that the generating function for the Donaldson invariants is an analytic function
$q : H^2(X,\mathbb{R}) \to \mathbb{R}$ constructed from the intersection form of the manifold, and a finite number
of certain characteristic classes in $H^2(X,\mathbb{Z})$.

A general framework providing a complete evaluation was later established 
in \cite{MooreWitten}: Moore and Witten obtained the generating function for the correlation functions as a regularized integral over the $u$-plane, henceforth called the $u$-plane integral.
The integrand is a modular invariant function which is determined by the elliptic surface and the gauge group. 
The regularization procedure defines a way of extracting from the integrand certain contributions for each boundary component near the cusps of the modular elliptic surface at $\tau=0,2,\infty$.
Moore and Witten observed that for $b_2^+=1$ the cuspidal
contributions at $\tau=0,2$ to the $u$-plane integral coincided with the generating function for the 
Seiberg-Witten invariants. They went further and made the following
comprehensive conjecture:

\begin{conjecture}[Moore and Witten \cite{MooreWitten}]
\label{MooreWitten}
The contribution from the cusp at $\tau=\infty$ to the regularized $u$-plane integral 
of a smooth, compact, oriented, simply connected four-manifold $X$ without boundary
and $b_2^+=1$ is the generating function for the Donaldson invariants of $X$.
\end{conjecture}
\noindent
Based on the work in \cite{MooreWitten, EllingsrudGoettsche, Goettsche, GoettscheZagier, MalmendierOno} 
it follows (cf. Theorem \ref{MainTheorem_MO} and Corollaries \ref{MW_CP2bu} and \ref{MW_CP1xCP1}) that the only case of Conjecture \ref{MooreWitten} that remains still open is the following:
\begin{conjecture}
\label{conjecture_light}
The contribution from the cusp at $\tau=\infty$ to the regularized $u$-plane integral 
of $\mathbb{C}\mathrm{P}^2$ for the gauge group $\mathrm{SU}(2)$
is equal to the generating function for the $\mathrm{SU}(2)$-Donaldson invariants of $\mathbb{C}\mathrm{P}^2$.
\end{conjecture}
\noindent
For $\mathbb{C}\mathrm{P}^2$, the regularization procedure employed in the definition of the $u$-plane integral depends on interpreting the integrand as a total derivative, combined with constant term contributions from cusps. The integration by parts naturally introduces non-holomorphic modular forms of weight $3/2$ for the gauge group $\mathrm{SU}(2)$. As evidence for Conjecture \ref{conjecture_light}, Moore and Witten \cite{MooreWitten} computed the first 40 coefficients of the $u$-plane integral and found them to be in agreement with the Donaldson invariants previously determined by Ellingsrud and G\"ottsche \cite{EllingsrudGoettsche}. However, Conjecture \ref{conjecture_light} remains open. 
The same technique was later used in \cite{MalmendierOno} to evaluate the $u$-plane integral of $\mathbb{C}\mathrm{P}^2$ for the gauge group $\mathrm{SO}(3)$ using  non-holomorphic modular forms of weight $1/2$ and prove Conjecture \ref{MooreWitten} in this case.  This article is continuation of the work in \cite{MalmendierOno}, but for another four-manifold with $b_2^+=1$ which is not of simple type. More generally speaking, if the conjecture about the equivalence of the simple-type-condition
and $b_2^+>1$ is true, then the conjecture \ref{MooreWitten} implies that manifolds which are not of simple type (i.e., the ones for which the beautiful structure theorem of Kronheimer and Mrowka does not hold) have mock modular forms of half-integral weight as the generating functions for their Donaldson invariants instead.

At this point, we like to make an historic remark: mock theta functions first appeared in a letter of Ramanujan to Hardy in 1920. In his letter, Ramanujan listed several examples of functions that he called mock theta functions. These function have an asymptotic expansion at the cusps, similar to that of modular forms of weight $1/2$, possibly with poles at cusps, but cannot be expressed in terms of ordinary theta functions. The theory behind mock theta functions remained unclear until Sander Zwegers discovered their connection to harmonic weak Maass forms in 2001. We refer to the article \cite{Ono} and the references therein for a more detailed overview over the development and history of mock modular forms.
 
In modern terminology, a mock modular form is the holomorphic part of a harmonic weak Maass form, and a mock theta function is a mock modular form of weight $1/2$. A harmonic weak Maass form is a smooth complex-valued function on the upper half plane that transforms like a modular form of weight $k$, though it may not be holomorphic at cusps, and is harmonic with respect to the hyperbolic weight-$k$ Laplacian. It is also common to impose the condition that the Maass form grows at most exponentially fast at cusps which for mock modular forms means that they are meromorphic at cusps. It then follows from the definition that a mock modular form is holomorphic but not quite modular, while the harmonic weak Maass form is modular but not quite holomorphic. The space of mock modular forms of weight $k$ contains the space of modular forms that may be meromorphic at cusps of weight $k$ as a subspace. 

Moreover, since any harmonic weak Maass form of weight $k$ is annihilated by the
hyperbolic weight-$k$ Laplacian, the anti-holomorphic part of a harmonic weak
Maass form is in turn related to a holomorphic modular form of weight $2-k$ through a complex anti-linear, first order differential operator. In fact, this map constitutes an isomorphism between
the anti-holomorphic part of harmonic weak Maass forms and the space of holomorphic modular forms of weight $2 - k$. The weight-$(2 - k)$ modular form corresponding to a mock modular form is called its shadow \cite{ZagierBourbaki}.
Conversely, given a shadow the non-holomorphic part of a harmonic weak Maass form is
obtained through a period-integral from its shadow by inverting the differential operator.

In this article we will compute the $u$-plane integral 
for $\widehat{\mathbb{C}\mathrm{P}^2}=\mathbb{C}\mathrm{P}^2\# \,\overline{\mathbb{C}\mathrm{P}^2}$,
the blowup  of the complex projective plane in one point, and the Cartesian product $\mathbb{C}\mathrm{P}^1 \times \mathbb{C}\mathrm{P}^1$ directly in a chamber where the elliptic unfolding technique fails to work. 
This allows us to determine explicit formulas for their $\mathrm{SU}(2)$ and $\mathrm{SO}(3)$-Donaldson invariants
in terms of mock modular forms.
The generating functions for the Donaldson invariants can then all be expressed in terms of
a series of rational polynomials in a complex variable $\mu$. For $a,b \in \lbrace 0,1\rbrace$ and
$m,n \in \mathbb{N}_0$ we define
\begin{equation}
\label{coeff}
\widehat{D}^{\,ab}_{mn} = \sum_{k=0}^n  \left[ R^{\,ab}_{mnk} \; 
e^{-\mu^2 \, T} \; \frac{\vartheta_{ab}\left(\frac{\mu}{2\pi h} \Big| \tau \right)}{\vartheta_4(\tau)}
\; \mathcal{E}^k[Q_{ab}^+(\tau)] \right]_{q^0}
\end{equation}
where the modular functions $T$, $R^{\,ab}_{mnk}$, and $\mathcal{E}^k[Q_{ab}(\tau)]$ will be defined in (\ref{qforms}),
(\ref{R_mnk}), and (\ref{EkQab}) respectively. 
$Q_{ab}^+(\tau)$ is the holomorphic part of the harmonic Maass form of weight 3/2 and 1/2 whose shadow is $\vartheta_{ab}(0|\tau)$ and $\eta^3(\tau)$ respectively. (For the definition of the Jacobi theta functions and Dedekind eta function see Table \ref{JacobiTheta}).
Specifically, for $a=0$ or $b=0$, $Q_{ab}(\tau)$ is one of Zagier's weight 3/2 Maass-Eisenstein series \cite{Zagier}, and for $a=1$ and $b=1$, $Q_{ab}(\tau)$ is the weight 1/2 harmonic Maass form described using a non-holomorphic Jacobi form first constructed by Zwegers \cite{Z2}.

To state our results we need to label the generators of the homology group $H_2(X,\mathbb{Z})$
and its Poincar\'e duals. For $\mathbb{C}\mathrm{P}^1 \times \mathbb{C}\mathrm{P}^1$ we denote by $\mathrm{F}$ and 
$\mathrm{G}$ the Poincar\'e duals of the classes $\mathrm{f}, \mathrm{g}$ of the fibers for the projections onto its factors. It follows that $\mathrm{F}^2=\mathrm{G}^2=0$ and $(\mathrm{F},\mathrm{G})=1$.
Similarly, for $\widehat{\mathbb{C}\mathrm{P}^2}=\mathbb{C}\mathrm{P}^2 \# \,\overline{\mathbb{C}\mathrm{P}^2}$ we denote
by $\mathrm{H}$ the first Chern class of the dual of the hyperplane bundle over 
$\mathbb{C}\mathrm{P}^2$ and by $\mathrm{E}$ the class of the exceptional divisor on the blowup.
It follows that $\mathrm{H}^2=-\mathrm{E}^2=1$ and $(\mathrm{H},\mathrm{E})=0$.
We denote the Poincar\'e duals of $\mathrm{H}$ and $\mathrm{E}$ by $\mathrm{h}$ and $\mathrm{e}$
respectively.

\noindent
Our main result is the following:
\begin{theorem}
\label{main_prop}
\begin{enumerate}
\item[]
\item
On $X= \mathbb{C}\mathrm{P}^2\# \,\overline{\mathbb{C}\mathrm{P}^2}$
let $\omega =\mathrm{H} - \epsilon \, \mathrm{E}$ with $0< \epsilon \ll 1$ be the period point of the metric.  The generating function for the $\mathrm{SO}(3)$-Donaldson invariants 
in the variables $p \, \mathrm{x} \in H_0(X, \mathbb{Z})$ and $S=\kappa \,\mathrm{h} 
+ \mu \, \mathrm{e}\in H_2(X, \mathbb{Z})$ is
\begin{equation*}
 Z = \sum_{m,n \in \mathbb{N}_0} \dfrac{p^m}{m!} \dfrac{\kappa^{2n}}{(2n)!} \; \widehat{D}^{\,11}_{mn} \;.
\end{equation*}
\item Assuming Conjecture \ref{conjecture_light} the generating function for the Donaldson invariants of $\mathbb{C}\mathrm{P}^2\# \,\overline{\mathbb{C}\mathrm{P}^2}$ for the gauge group $\mathrm{SU}(2)$  is
\begin{equation*}
 Z = \sum_{m,n \in \mathbb{N}_0} \dfrac{p^m}{m!} \dfrac{\kappa^{2n+1}}{(2n+1)!} \; \widehat{D}^{\,01}_{mn} \;.
\end{equation*}
\item On $X=\mathbb{C}\mathrm{P}^1 \times \mathbb{C}\mathrm{P}^1$
let $\omega= \frac{1}{2} \,\mathrm{F} +  \mathrm{G}$ be the period point of the metric.
The generating function for the $\mathrm{SU}(2)$-Donaldson invariants 
in the variables $p \, \mathrm{x} \in H_0(X, \mathbb{Z})$ and
$S=\kappa_f \, \mathrm{f} + 2\, \kappa_g \, \mathrm{g} \in H_2(X, \mathbb{Z})$
is
\begin{equation*}
 Z_{\tau=\infty}
  =  \frac{1}{2} \, \sum_{a,b \,\in \lbrace 0,1\rbrace}
  \; \sum_{m,n \in \mathbb{N}_0} \dfrac{p^m}{m!} \;\dfrac{\kappa^{2n-ab+1}}{(2n-ab+1)!} \;  (-1)^{ab} \; \widehat{D}^{\,ab}_{mn} 
\end{equation*}
where $\kappa= \kappa_f + \kappa_g$ and $\mu=-\kappa_f + \kappa_g$.
\item Assuming Conjecture \ref{conjecture_light} the generating function for the Donaldson invariants of $\mathbb{C}\mathrm{P}^1 \times \mathbb{C}\mathrm{P}^1$ for the gauge group $\mathrm{SO}(3)$  is
\begin{equation*}
Z_{\tau=\infty}
  = \frac{1}{2} \, \sum_{a,b\,\in \lbrace 0,1\rbrace}
  \;  \sum_{m,n \in \mathbb{N}_0} \dfrac{p^m}{m!} \;  \dfrac{\kappa^{2n-ab+1}}{(2n-ab+1)!} \; (-1)^{(a+1)b} \; \widehat{D}^{\,ab}_{mn} \;.
\end{equation*}
\end{enumerate}
\end{theorem}
As stated in Part (1) and (2) of Theorem \ref{main_prop} the Donaldson invariants of the blowup
$X=\mathbb{C}\mathrm{P}^2\# \,\overline{\mathbb{C}\mathrm{P}^2}$ only make use of the coefficients
(\ref{coeff}) for $b=1$. As we shall see the reason is that $X$ is not a
spin manifold. The original motivation for this article was to answer the question whether
the coefficients (\ref{coeff}) have a geometric meaning in the case $b=0$ as well.
Parts~(3) and (4) of Theorem \ref{main_prop} give a positive answer to this question.
This means that Zagier's weight 3/2 Maass-Eisenstein series $Q_{ab}(\tau)$ for $b=0$ 
used in Equation~(\ref{coeff}) are the generating functions for certain Donaldson invariants.
The key technique in the proof of Theorem \ref{main_prop} is what is known in string theory
as `summing over all spin structures of the torus'.

This article is structured as follows.
In Section~\ref{donaldson} we recall the definition and basic properties of the Donaldson invariants. In Section~\ref{uplane} we define the $u$-plane integral and explain two of its fundamental properties, the 
so-called wall-crossing and blowup formulas. In Section~\ref{complex_surfaces} we evaluate the $u$-plane 
integral for the simplest complex surfaces: the complex projective plane $\mathbb{C}\mathrm{P}^2$, the blowup  of the complex projective plane in one point $\mathbb{C}\mathrm{P}^2 \# \,\overline{\mathbb{C}\mathrm{P}^2}$,  and the Cartesian product $\mathbb{C}\mathrm{P}^1 \times \mathbb{C}\mathrm{P}^1$. To carry out this computation, we employ the theory of harmonic Maass forms. 
Specifically, we relate the relevant generating functions  for the $\mathrm{SU}(2)$-gauge theory and the $\mathrm{SO}(3)$-gauge theory to the holomorphic parts of harmonic Maass forms of weight 3/2 and 1/2
respectively. In Section \ref{proof} we combine these results to prove Theorem \ref{main_prop}.
Throughout the article we will assume that the reader is familiar with the results in \cite{MalmendierOno}.
More details in terms of the relevant background in number theory and differential geometry can also be found there.

\medskip

\noindent
{\small
{\bf Notation:} In this article we will use the following definition for the Jacobi theta function 
\begin{equation}
 \vartheta_{ab}(v|\tau) = \sum_{n\in\mathbb{Z}} q^{\frac{(2n+a)^2}{8}} \; e^{\pi i \,(2n+a)(v+\frac{b}{2})}
\end{equation}
where $a,b \in \lbrace 0,1\rbrace$, $v\in\mathbb{C}$, $q=\exp(2\pi i \tau)$,
$\tau = x+iy \in \mathbb{H}$, and $\mathbb{H}$ is the complex upper half-plane.
The relation to the standard Jacobi theta functions is summarized in the following table:
\begin{equation}
\label{JacobiTheta}
 \begin{array}{l|l|l}
  \vartheta_1(v|\tau) = \vartheta_{11}(v|\tau)
& \vartheta_1(0|\tau) = 0
& \vartheta_1'(0|\tau) = - 2\pi \eta^3(\tau) \\ [1ex]
\hline && \\[-2ex]
  \vartheta_2(v|\tau) = \vartheta_{10}(v|\tau) 
& \vartheta_2(0|\tau) = \sum_{n\in\mathbb{Z}} q^{\frac{(2n+1)^2}{8}} 
& \vartheta_2'(0|\tau) = 0 \\ [1ex]
\hline && \\[-2ex]
  \vartheta_3(v|\tau) = \vartheta_{00}(v|\tau)
& \vartheta_3(0|\tau) = \sum_{n\in\mathbb{Z}} q^{\frac{n^2}{2}}
& \vartheta_3'(0|\tau) = 0 \\ [1ex]
\hline && \\[-2ex]
  \vartheta_4(v|\tau) = \vartheta_{01}(v|\tau)
& \vartheta_4(0|\tau) = \sum_{n\in\mathbb{Z}} (-1)^n \, q^{\frac{n^2}{2}}
& \vartheta_4'(0|\tau) = 0
\end{array}
\end{equation}
Here $\eta(\tau)$ is the Dedekind eta function with
\begin{equation}
 \eta^3(\tau) = \sum_{n=0}^\infty (-1)^n \; (2n+1) \; q^{\frac{(2n+1)^2}{8}} \;.
\end{equation}
We will also use the notation $\vartheta_j=\vartheta_j(0|\tau)$ for $j=2,3,4$.}

\section{Donaldson theory of simply connected four-manifolds}
\label{donaldson}

The Donaldson invariants of a smooth, compact, oriented, simply connected Riemannian four-manifold
$(X,g)$ without boundary are defined by using intersection theory on the moduli space of 
anti-self-dual instantons for the gauge group $\mathrm{SU}(2)$ or $\mathrm{SO}(3)$ respectively \cite{Goettsche2}.
Given a homology orientation some cohomology classes on the instanton moduli space can be associated to homology classes of $X$ through the slant
product and then evaluated on a fundamental class. Define $\mathbf{A}(X)={\rm Sym}(H_0(X,\mathbb{Z}) \oplus H_2(X,\mathbb{Z}))$ and
regard the Donaldson invariants as the functional
\begin{equation}
  \mathcal{D}_{w_2(E)}^{X,g}: \mathbf{A}(X) \rightarrow \mathbb{Q} \;,
\end{equation}
where $w_2(E) \in H^2(X, \mathbb{Z}_2)$ is the second Stiefel-Whitney class
of the gauge bundles which are considered. 
Since $X$ is simply connected there is an integer class
$2\lambda_0 \in H^2(\mathbb{C}\mathrm{P}^2,\mathbb{Z})$ that is not divisible by two and
whose mod-two reduction is $w_2(E)$.
Let $\{\mathrm{s}_i\}_{i=1,\ldots, b_2}$ be a basis of the two-cycles of $X$. We introduce
the formal sum $S=\sum_{i=1}^{b_2} \kappa^i \,\mathrm{s}_i$ where  $\kappa^i$ are complex
numbers. The generator of the zero-class of $X$ will be denoted by $\mathrm{x} \in H_0(X, \mathbb{Z})$. 
The Donaldson-Witten generating function is
\begin{equation}
\label{donwi}
 Z_{DW}(p, \kappa)=\mathcal{D}^{X,g}_{w_2(E)}(e^{p \, \mathrm{x} + S}) \;,
\end{equation}
so that the Donaldson invariants are read off from the expansion 
of (\ref{donwi}) as the coefficients of powers of $p$ and $\kappa=(\kappa^1, \dots, \kappa^{b_2})$. 

The Donaldson invariants are topological invariants of $X$ and do not depend on the metric $g$ if $b_2^+ >1$.
For $b_2^+=1$ the Donaldson invariants are no longer independent of the metric \cite{KotschickMorgan}.
A metric $g$ on $X$ determines a ray within the set of self-dual (with respect
to $g$) harmonic two-forms $H^2(X,\mathbb{R})^+=\{ \alpha \in H^2(X,\mathbb{R}) | \, \alpha^2>0\}$.
The choice of an homology orientation amounts to choosing a connected component of
$H^2(X,\mathbb{R})^+/\mathbb{R}^+$. A representative for such a ray is given by a normalized self-dual two-form (or period point) $\omega$ with $\omega^2=1$.
We will always assume that a chosen period point is located in the component determined by the homology orientation.
The generating function (\ref{donwi}) for the Donaldson invariants of a manifold $X$ with $b_2^+=1$ depends on the metric 
through the position of the period point $\omega$ in $H^2(X,\mathbb{R})^+$ via a systems of walls and chambers. In fact, the generating function (\ref{donwi}) has a discontinuous variation in $\omega$
if a cohomology class $\lambda\in H^{2}(X,\mathbb{Z})+\lambda_0$ is such that the period
$\omega\cdot\lambda$ changes its sign. We then say that $\lambda$ defines a wall. The chambers are 
the complements of these walls.

\section{The $u$-plane integral}
\label{uplane}
From now on we will assume that $(X,g)$ is a smooth, compact, oriented, simply connected Riemannian four-manifold without boundary
and $b_2^+=1$. The $u$-plane integral $Z$ is a generating function in the variables $p$ and $\kappa$
whose coefficients are the integrals 
of certain modular forms over the fundamental domain of the group $\Gamma_{0}(4)$ and depend on the period point $\omega$, the lattice $H_2(X,\mathbb{Z})$ together with the intersection form $(.\, ,.)$, the second Stiefel-Whitney classes
of the gauge bundle $w_2(E)$ and the tangent bundle $w_2(X)$ whose integral liftings are denoted by $2\lambda_0$ and $w_2$ respectively. The $u$-plane integral is non-vanishing only for manifolds with $b_2^{+}=1$.
The explicit form of $Z$ for simply connected four-manifolds was first introduced in \cite{MooreWitten}.  
For the convenience of the reader we quickly review the explicit construction of the $u$-plane in this
chapter. Our approach to the $u$-plane integral, as well as its normalization follows closely the approach in \cite{LabastidaLozano,LozanoMarino,MarinoMoore}.

\noindent
We will denote the self-dual and anti-self-dual projections of any two-form $\lambda\in H^2(X,\,\mathbb{Z})+\lambda_0$
by $\lambda_{+}=(\lambda,\omega)\omega$ and $\lambda_{-}=\lambda-\lambda_{+}$ respectively.
We first introduce the integral
\begin{equation}
\label{newcpione}
\mathcal{G}(\rho)= \int_{\Gamma_0(4) \backslash \mathbb{H}}^{\text{reg}}
\dfrac{dx dy}{y^{\frac{3}{2}}} \; \widehat f (p,\kappa) \; \bar\Theta(\xi) \;.
\end{equation}
In this expression $\widehat f(p,\kappa)$ is the almost holomorphic
modular form given by
\begin{equation}
\label{nwcnvpione}
\widehat f(p,\kappa)  =
\dfrac{\sqrt{2}}{64 \pi } \dfrac{\vartheta_4^{\sigma}}{h^3 \cdot f_{2}} \; 
 e^{2 \, p \, u + S^2 {\widehat T}}
\end{equation} 
where $\sigma$ is the signature of $X$ and $S^2=(S,S)=\sum_{i,j} \kappa^i \kappa^j (\mathrm{s}_i,\mathrm{s}_j)$. $\bar{\Theta}$ is the Siegel-Narain theta function and is defined to be
\begin{equation}
\label{siegnar}
\begin{split}
\bar \Theta(\xi)\,   = \, &\exp\left\lbrack \frac{\pi}{2\,y} \Big(\bar{\xi}_{+}^2-\bar{\xi}_{-}^2\Big)\right\rbrack \\
\times  \sum_{\lambda\in H^2+ \lambda_0 }
& \exp\Big\lbrack - i \pi \bar\tau (\lambda_+)^2 - i \pi   \tau(\lambda_-  )^2
- 2 \pi i\, (\lambda,\bar\xi) + \pi i\, (\lambda ,  w_2)\Big\rbrack 
\end{split}
\end{equation}
where $\bar{\xi}= \bar{\xi}_+ + \bar{\xi}_-$, $\bar{\xi}_+=\rho \, y \, h \, \omega$, 
$\bar{\xi}_-= S_-/(2\pi h)$, and $\rho \in \mathbb{R}$.
The Siegel-Narain theta function only depends on the lattice data $(H^2(X), \omega, \lambda_0, w_2)$.
We have denoted the intersection form in two-cohomology by $(.\, , .)$ and
used Poincar\'e duality to convert cohomology classes into homology classes.
In the above expressions $u$, $T$, $h$, and $f_{2}$
are the modular forms defined as follows:
\begin{equation}
\label{qforms}
\begin{array}{rclcrcl}
u&=&  \dfrac{\vartheta_2^4+\vartheta_3^4}{2 \, (\vartheta_2\vartheta_3)^2} \;, &\qquad&
h&=& \frac{1}{2} \, \vartheta_2 \, \vartheta_3 \;, \\ [2ex]
T&=&-\dfrac{1}{24} \left(  \dfrac{E_2}{h^2} - 8 \, u\right) \;, &&
f_{2}&=&\dfrac{\vartheta_2 \, \vartheta_3}{2 \, \vartheta_4^8} \;.
\end{array}
\end{equation}
$T$ does not transform well under modular transformations, due to the presence of the second normalized Eisenstein series $E_2=E_2(\tau)$ with
\begin{displaymath}
E_2(\tau)=1-24\sum_{n=1}^{\infty}\sum_{d\mid n}d \cdot q^n \;.
\end{displaymath}
Therefore, in Equation (\ref{nwcnvpione})
we have used the related form $\widehat{T} = T + 1/(8\pi y h^{2})$
which is not holomorphic but transforms well under modular transformations.
We also define the related holomorphic function $f(p, \kappa)$
as in Equation (\ref{nwcnvpione}), but with $T$ instead of $\widehat{
T}$. The $u$-plane integral is defined to be
\begin{equation}
\label{u-plane_integral}
 Z\Big(X,\omega,\lambda_0,w_2\Big) =\left. \left\lbrack (S,\omega) + 2  \dfrac{d}{d\rho} \right\rbrack\right|_{\rho=0} \; \mathcal{G}(\rho)\;.
\end{equation}
If there is no danger of confusion we suppress the arguments $(X,\omega,\lambda_0,w_2)$ of $Z$.
\begin{remark}
The definition of the $u$-plane integral in Equation (\ref{u-plane_integral}) agrees with the definition given
in \cite{LozanoMarino}. However, compared to the original definition in \cite{MooreWitten} a factor of 
$\exp{[2 \pi i (\lambda_0,\lambda_0) + \pi i (\lambda_0, w_2)]}$ is missing. For all cases considered in this
article this factor is equal to one. 
\end{remark}
The regularization procedure applied in the definition of the integral 
(\ref{newcpione}) was described in detail in \cite{MooreWitten}. It defines 
a way of extracting certain contributions for each boundary component near the cusps of $\Gamma_0(4)\backslash\mathbb{H}$. 
Since the cusps are located at $\tau=\infty$, $\tau=0$, and $\tau=2$ we obtain
$Z$ as the sum of these contributions from the cusps:
\begin{equation}
 Z = Z_{\tau=0} + Z_{\tau=2} + Z_{\tau=\infty} \;.
\end{equation}
For the complex surfaces considered in this article we will show in Theorem \ref{evaluation1} and
Corollary \ref{evaluation2} that the regularization procedure in the integral amounts to 
computing the constant coefficient term in the series expansion of the integrand (\ref{newcpione}).

\subsection{Wall-crossing for the $u$-plane integral}
\label{wall_crossing}
\noindent
The integral (\ref{newcpione}) has a discontinuous variation in $\omega$ at the cusps 
of $\Gamma_0(4) \backslash \mathbb{H}$ if for the cohomology class $\lambda\in H^{2}(X,\mathbb{Z})+\lambda_0$ the period $\omega\cdot\lambda$ changes sign. 
The conditions for wall-crossing are $\lambda^2<0$ and $\lambda_{+}=0$.
The wall-crossing of the $u$-plane integral associated with the cusp at infinity $\tau= \infty$ was first derived 
in Section $4$ of \cite{MooreWitten}:
\begin{theorem}[Wall-crossing]
Let
\begin{equation}
\label{Z_diff}
 Z_{\tau=\infty}(X,\omega_1,\lambda_0, w_2) - Z_{\tau=\infty}(X,\omega_2,\lambda_0, w_2) 
 = \sum_{\lambda} \; WC(\lambda)
\end{equation}
be the difference between the cusp contribution at $\tau=\infty$ for the period points $\omega_1$ and $\omega_2$.
The sum is understood to run over all $\lambda \in H^2(X,\mathbb{Z})+\lambda_0$ with $\omega_1\cdot\lambda>0>\omega_2\cdot\lambda$.
We then have
\begin{equation}
\label{wcross}
\begin{split}
WC(\lambda)
& = -\frac{i}{2} (-1)^{ (\lambda-\lambda_0,   w_2)} \; e^{2\pi i\lambda_0^2} \\
&\times \left\lbrack q^{- \frac{\lambda^2}{2}} \, \dfrac{\vartheta_4^\sigma}{h^2 \cdot f_{2}} \;
\exp\Big\lbrack 2\,p\, u  +  S^2 \, T -  \frac{i}{h} (\lambda , S) \Big\rbrack \right\rbrack_{q^0} \;
\end{split}
\end{equation}
where $u, h, T, f_2$ were defined in (\ref{qforms}) and $p \, \mathrm{x} \in H_0(X, \mathbb{Z})$, 
$S \in H_2(X, \mathbb{Z})$. 
\end{theorem}
\begin{proof}
Using the $q$-expansion of the different modular forms, it is
easy to check that the wall-crossing term is different from zero only if
$0>\lambda^2\ge p_1/4$ where $p_1$ is the Pontryagin number of the gauge
bundle (and $p_1 \equiv w_2(E)^2$ mod $4$). By combining the contributions (\ref{wcross}) 
for all crossed walls one finds that the difference between the cusp contribution at $\tau=\infty$ for $\omega_1$ and $\omega_2$ is given by Equation (\ref{Z_diff}) and (\ref{wcross}).
\end{proof}

\begin{remark}
The expression (\ref{wcross}) agrees with the wall-crossing formula for the Donaldson invariants of
non-simply connected manifolds with $b_2^+=1$ derived in \cite{Goettsche} under the assumption of the Kotschick-Morgan conjecture.
The equality of the two wall-crossing formulas means that for manifolds with $b_2^+=1$ and $b_2^-\ge 1$ the Moore-Witten conjecture can be proved effectively by showing that
the generating function for the Donaldson invariants and the $u$-plane integral agree in one particular chamber.
\end{remark}

\subsection{Blowup formulas for the $u$-plane integral}
\label{blow_up}
The blowup formula relates the $u$-plane integral of a four-manifold $X$ with $b_2^+=1$ to the 
$u$-plane integral of the connected sum $\widehat{X}=X \# \,\overline{\mathbb{C}\mathrm{P}^2}$ of $X$ and $\mathbb{C}\mathrm{P}^2$ with the opposite orientation. In fact, the blowup formula expresses the coefficients of the generating function of $\widehat{X}$ in terms of the coefficients
of the generating function of $X$.

Let $\mathrm{E}$ be the class of the exceptional divisor such that $H^2(\widehat{X},\mathbb{R})=H^2(X,\mathbb{R}) \oplus \mathbb{R} \,
\mathrm{E}$ and $H^2(X,\mathbb{R})$ is identified with the classes in $H^2(\widehat{X},\mathbb{R})$ orthogonal to $\mathrm{E}$.
Let $\mathrm{e}$ be the Poincar\'e dual of $\mathrm{E}$. We need to choose the metric on $\widehat{X}$ to be very close to the 
pullback of the metric on $X$ to make the computation of Moore and Witten applicable.
Thus, if $\omega \in H^2(X,\mathbb{R})^+$ denotes the period point of $X$ we choose the period point $\omega+ := \omega - \epsilon \, \mathrm{E}$ with $0<\epsilon \ll 1$ for $\widehat{X}$. We also denote the integral lifting of $w_2(\widehat{X})$ by $\widehat{w}_2$.
The cycles $\{\mathrm{s}_1, \dots, \mathrm{s}_{b_2}, \mathrm{e}\}$ form a basis of the two-cycles of $\widehat{X}$.
Thus, the formal sum $\widehat{S}=\sum_{i=1}^{b_2} \kappa^i \,\mathrm{s}_i + \mu \, \mathrm{e}$ will be appearing in the $u$-plane integral where $\mu$ is a complex variable. The result of \cite[Section 6]{MooreWitten} can now be stated as follows:
\begin{theorem}[Blowup formula for $\mathrm{SO}(3)$]
\label{blowupSO3}
For $p \, \mathrm{x} \in H_0(\widehat{X}, \mathbb{Z})$ and $\widehat{S}=S + \mu \, \mathrm{e} \in H_2(\widehat{X},\mathbb{Z})$ we have that
\begin{equation}
\label{blowup_SO3}
 Z \Big( \widehat{X}, \omega+, \widehat{\lambda}_0 = \lambda_0 + \frac{1}{2} \, \mathrm{E}, \widehat{w}_2 =w_2 + \mathrm{E} \Big) = 
 \left. \left\lbrack (S,\omega) + 2  \dfrac{d}{d\rho} \right\rbrack\right|_{\rho=0} \; \widehat{\mathcal{G}}_0(\rho)\;,
\end{equation}
where
\begin{equation}
\label{newcpione_0}
\widehat{\mathcal{G}}_0(\rho)= \int_{\Gamma_0(4) \backslash \mathbb{H}}^{\text{reg}}
\dfrac{dx dy}{y^{\frac{3}{2}}} \; \; \widehat f (p,\kappa) \; \bar\Theta(\xi) \; \; \; \frac{1}{2\sqrt{2}} \, e^{-\frac{u\, \mu^2}{3}} \, \sigma(2\sqrt{2} \,\mu) \;.
\end{equation}
$\sigma(2\sqrt{2} \,\mu)$ is the Weierstrass sigma function for the periods $2\,\underline{\omega}=4\sqrt{2}\pi h$ and $2\,\underline{\omega}'=\tau \, 4\sqrt{2}\pi h$. 
$\hat{f}$ and $\bar{\Theta}$ depend on $X$ only and were defined in (\ref{nwcnvpione}) and
(\ref{siegnar}) respectively. The quantities $u$ and $h$ were defined in (\ref{qforms}).
\end{theorem}
\begin{proof}
On $\widehat{X}$ the $u$-plane integral receives an additional factor of
\begin{equation}
 - \, \dfrac{\vartheta_1\left( \frac{\mu}{2\pi h} \Big| \tau \right)}{\vartheta_4(0|\tau)} \; e^{-\mu^2 \, T} \;.
\end{equation}
To prove Equation (\ref{newcpione_0}) one uses the identity
\begin{equation}
 \dfrac{\sigma(z)}{z} = \exp{\left(\dfrac{\pi^2 \, E_2 \, z^2}{24 \, \underline{\omega}^2}\right)} \;\, \dfrac{\vartheta_1(v|\tau)}{v \, \vartheta_1'(0|\tau)} 
\end{equation}
for $v=z/(2\underline{\omega})$ in the definition of the $u$-plane integral and
$2 \, \eta^3 = \vartheta_2 \vartheta_3 \vartheta_4$.
\end{proof}

The additional factor in Equation (\ref{newcpione_0}) has a series expansion in $\mu$ 
whose coefficients are rational polynomials in $u$ since
\begin{equation}
 \sigma(z) = \sum_{m,n=0}^\infty a_{mn} \; \left(\frac{1}{2}g_2\right)^m  \; (2\, g_3)^n \; \dfrac{z^{4m+6n+1}}{(4m+6n+1)!}
\end{equation}
where $g_2=u^2/12-1/16$, $g_3=u^3/216-u/192$. The coefficients $a_{m,n} \in \mathbb{Q}$ can be found in \cite[18.5.8]{AbramowitzStegun}.
The first terms are
\begin{equation}
\label{ws}
\begin{split} 
& \frac{1}{2\sqrt{2}} \, e^{-\frac{u\, \mu^2}{3}} \, \sigma(2\sqrt{2} \,\mu)\\
= & \; \mu - (2u) \, \frac{\mu^3}{3!} +  \Big( (2 \, u)^2 + 2 \Big) \, \frac{\mu^5}{5!} - \Big( (2u)^3 + 6 \, (2u) \Big) \frac{\mu^7}{7!} + O(\mu^{11}) \;. 
\end{split}
\end{equation}
\begin{remark}
The blowup function $\frac{1}{2\sqrt{2}} \, e^{-\frac{u\, \mu^2}{3}} \, \sigma(2\sqrt{2} \,\mu)$ agrees with the blowup function for the 
$\mathrm{SO}(3)$-Donaldson invariants derived in \cite{FintushelStern}. The expansion (\ref{ws}) agrees with the relations $\mathcal{D}^{\widehat{X},\omega+}_{w_2+\mathrm{E}}(\mathrm{z} \, \mathrm{e}^{2k})=0$ for $k \in \mathbb{N}$ and
for $\mathrm{z}\in \mathbf{A}(X)$
\begin{equation*}
\begin{split}
\mathcal{D}^{\widehat{X},\omega+}_{w_2+\mathrm{E}}(\mathrm{z} \, \mathrm{e}) & =  \mathcal{D}^{X,\omega}_{w_2}(\mathrm{z})\;, \\
\mathcal{D}^{\widehat{X},\omega+}_{w_2+\mathrm{E}}(\mathrm{z} \, \mathrm{e}^3) & = - \mathcal{D}^{X,\omega}_{w_2}(\mathrm{z}\, \mathrm{x})\;, \\
\mathcal{D}^{\widehat{X},\omega+}_{w_2+\mathrm{E}}(\mathrm{z} \, \mathrm{e}^5) & = \mathcal{D}^{X,\omega}_{w_2}(\mathrm{z} \, \mathrm{x}^2)
+ 2 \, \mathcal{D}^{X,\omega}_{w_2}(\mathrm{z})\;, \\
\mathcal{D}^{\widehat{X},\omega+}_{w_2+\mathrm{E}}(\mathrm{z} \, \mathrm{e}^7) & =  - \mathcal{D}^{X,\omega}_{w_2}(\mathrm{z} \, \mathrm{x}^3)
- 6 \, \mathcal{D}^{X,\omega}_{w_2}(\mathrm{z}) \;.
\end{split}
\end{equation*}
for the Donaldson invariants. The latter relations were first derived in \cite{Leness}.
\end{remark}

\noindent
We have a similar result for the $u$-plane integral for the gauge group $\mathrm{SU}(2)$:
\begin{theorem}[Blowup formula for $\mathrm{SU}(2)$]
\label{blowupSU2}
For $p \, \mathrm{x} \in H_0(\widehat{X}, \mathbb{Z})$ and $\widehat{S}=S + \mu \, \mathrm{e} \in H_2(\widehat{X},\mathbb{Z})$ we have that
\begin{equation}
\label{blowup_SU2}
 Z \Big( \widehat{X}, \omega+, \widehat{\lambda}_0 = \lambda_0, \widehat{w}_2=w_2 + \mathrm{E} \Big) = 
 \left. \left\lbrack (S,\omega) + 2  \dfrac{d}{d\rho} \right\rbrack\right|_{\rho=0} \; \widehat{\mathcal{G}}_{3}(\rho)\;,
\end{equation}
where 
\begin{equation}
\label{newcpione_j}
\widehat{\mathcal{G}}_{3}(\rho)= \int_{\Gamma_0(4) \backslash \mathbb{H}}^{\text{reg}}
\dfrac{dx dy}{y^{\frac{3}{2}}} \; \; \widehat f (p,\kappa) \; \bar\Theta(\xi) 
\; \; \; e^{-\frac{u\, \mu^2}{3}} \, \sigma_3(2\sqrt{2} \,\mu) \;.
\end{equation}
$\sigma_3(2\sqrt{2} \,\mu)$ is the Weierstrass sigma function for the periods $2\,\underline{\omega}=4\sqrt{2}\pi h$ and $2\,\underline{\omega}'=\tau \, 4\sqrt{2}\pi h$, and the 
half-period $\underline{\omega}_3=\omega'$. 
$\hat{f}$ and $\bar{\Theta}$ depend on $X$ only and were defined in (\ref{nwcnvpione}) and
(\ref{siegnar}) respectively. The quantities $u$ and $h$ were defined in (\ref{qforms}).
\end{theorem}
\begin{proof}
For $\omega_j = (1-b) \, \omega + (1-a) \, \omega'$ with $a,b$ not both equal to $1$, it follows
\begin{equation}
  \frac{\vartheta_{1}( v + \frac{\underline{\omega}_j}{2\underline{\omega}} |\tau)}
  {\vartheta_{1}( \frac{\;\underline{\omega}_j}{2\underline{\omega}} |\tau)} = e^{-\delta_{a,0} \pi i v} \; 
\frac{\vartheta_{ab}( v |\tau)}{\vartheta_{ab}( 0 |\tau)} 
\end{equation}
where $v=z/(2\underline{\omega})$. The index $j$ is given in terms of $(a,b)$ by the map $j\leftrightarrow (a,b)$ with $1 \leftrightarrow (1,0), 2 \leftrightarrow (0,0), 3 \leftrightarrow (0,1)$. 
The following relation between the Jacobi theta functions and the Weierstrass sigma functions holds:
\begin{equation}
 \sigma_j(z) 
= \exp{\left(\dfrac{\pi^2 \, E_2 \, z^2}{24 \, \underline{\omega}^2}\right)} \; \frac{\vartheta_{ab}(v|\tau)}{\vartheta_{ab}(\tau)}\;.
\end{equation}
For $j\in \{1,2,3\}$ the function $\sigma_j(2\sqrt{2} \,\mu)$ is the Weierstrass sigma function for for the periods $2\,\underline{\omega}=4\sqrt{2}\pi h$ and $2\,\underline{\omega}'=\tau \, 4\sqrt{2}\pi h$, and the  half-period $\underline{\omega}_j$.
Setting $(a,b)=(0,1)$ and $j=3$, Equation (\ref{newcpione_j}) follows.
\end{proof}
The additional factor in Equation (\ref{newcpione_j}) has a series expansion in $\mu$ 
whose coefficients are rational polynomials in $u$. The first terms are
\begin{equation}
\label{ws3}
\begin{split}
e^{-\frac{u\, \mu^2}{3}} \, \sigma_3(2\sqrt{2}\mu) 
= 1 - 2 \, \frac{\mu^4}{4!} + 8 \, (2u) \, \frac{\mu^6}{6!} - \Big( 32 \, (2u)^2 + 4 \Big) \, \frac{\mu^8}{8!} + O(\mu^{10})\;.
\end{split}
\end{equation}
\begin{remark}
The blowup function $e^{-\frac{u\, \mu^2}{3}} \, \sigma_3(2\sqrt{2} \,\mu)$ agrees with the blowup function for the $\mathrm{SU}(2)$-Donaldson invariants derived in \cite{FintushelStern}. 
The expansion (\ref{ws3}) agrees with the relations $\mathcal{D}^{\widehat{X},\omega+}_{w_2}(\mathrm{z} \, \mathrm{e}^{2k-1})=0$ for $k \in \mathbb{N}$ and for $\mathrm{z}\in \mathbf{A}(X)$
\begin{equation*}
\begin{split}
\mathcal{D}^{\widehat{X},\omega+}_{w_2}(\mathrm{z}) & =  \mathcal{D}^{X,\omega}_{w_2}(\mathrm{z})\;, \\
\mathcal{D}^{\widehat{X},\omega+}_{w_2}(\mathrm{z} \, \mathrm{e}^4) & = - 2 \, \mathcal{D}^{X,\omega}_{w_2}(\mathrm{z})\;, \\
\mathcal{D}^{\widehat{X},\omega+}_{w_2}(\mathrm{z} \, \mathrm{e}^6) & = 8 \, \mathcal{D}^{X,\omega}_{w_2}(\mathrm{z} \, \mathrm{x}) \;, \\
\mathcal{D}^{\widehat{X},\omega+}_{w_2}(\mathrm{z} \, \mathrm{e}^8) & = - 32 \, \mathcal{D}^{X,\omega}_{w_2}(\mathrm{z} \, \mathrm{x}^2)
- 4 \, \mathcal{D}^{X,\omega}_{w_2}(\mathrm{z})\;.
\end{split}
\end{equation*}
for the Donaldson invariants. 
\end{remark}

\section{Invariants for some complex surfaces}
\label{complex_surfaces}

\subsection{The projective plane}
The Fubini-Study metric $g$ on $\mathbb{C}\mathrm{P}^2$ is a K\"ahler metric with the 
K\"ahler form $K= \frac{i}{2} g_{i\bar{\jmath}} \, dz^i \wedge dz^{\bar{\jmath}}$.
It follows that the first Chern class of the dual of the hyperplane bundle over 
$\mathbb{C}\mathrm{P}^2$ is $\operatorname{H} = K/\pi$. We then have that
$\int_{\mathbb{C}\mathrm{P}^2} \operatorname{H}^2 =1$, $c_1(\mathbb{C}\mathrm{P}^2)=3 \operatorname{H}$, and $p_1(\mathbb{C}\mathrm{P}^2)=3 \operatorname{H}^2$. 
The Poincar\'e dual $\operatorname{h}$ of $\operatorname{H}$ is a generator of the rank-one homology group
$H_2(\mathbb{C}\mathrm{P}^2,\mathbb{Z})$.  
We denote the integral lifting of $w_2(E)$ by $2\lambda_0=a \, \mathrm{H} \in H^2(\mathbb{C}\mathrm{P}^2,\mathbb{Z})$,
and the integral lifting of $w_2(\mathbb{C}\mathrm{P}^2)$ by $w_2=- b \, \mathrm{H}\in H^2(\mathbb{C}\mathrm{P}^2,\mathbb{Z})$. Notice that $b=1$ as $\mathbb{C}\mathrm{P}^2$ is not spin. However, all formulas we will write down will remain well-defined for $b=0$ as well. 
We have the following lemma:
\begin{lemma}
On $X=\mathbb{C}\mathrm{P}^2$ let $\omega=\mathrm{H}$ be the period point
of the metric. Let $2\,\lambda_0 = a \, \mathrm{H}$ with $a \in \{0,1\}$ be an integral lifting 
of $w_2(E)$ and $b=1$. For $(X,\omega,\lambda_0,w_2=-b \, \mathrm{H})$ 
the Siegel-Narain theta function is
\begin{equation}
\label{siegnar_CP2}
\begin{split}
\bar \Theta  =  \exp{\left(\frac{\pi}{2\,y} \, \bar{\xi}_+^2\right)} \; \overline{\vartheta_{ab}\Big((\xi_+,\mathrm{H}) \Big| \tau\Big)}\;
\end{split}
\end{equation}
where $\bar\xi=\bar\xi_+=\rho \, y\, h \, \omega$. 
\end{lemma}
\noindent
It was shown in \cite{MooreWitten} that for $\sigma=1$ and any value of $a$ and $b$ we have
\begin{equation}
Z_{\tau=0} = Z_{\tau=2} =0 \;,
\end{equation}
hence $Z = Z_{\tau=\infty}$.
The $u$-plane integral in Equation (\ref{u-plane_integral}) can be
expanded as follows
\begin{equation}
 Z_{\tau=\infty}  = \sum_{m,n \in \mathbb{N}_0} \dfrac{p^m}{m!} \; \dfrac{\kappa^{2n-ab+1}}{(2n-ab+1)!} \; D^{\, ab}_{mn} 
\end{equation}
where
\begin{equation}
\label{coefficients_preint}
D^{\, ab}_{mn} = -\frac{\sqrt{2}}{32\pi} \sum_{k=0}^n  \int^{\text{reg}}_{\Gamma_0(4)\backslash \mathbb{H}}  \frac{dx \, dy}{y^{\frac{3}{2}}} \; R^{\,ab}_{mnk} \; \widehat{E}_2^k\; \left[  \overline{\vartheta_{ab}(0 | \tau)} - 4 \, y \,  \overline{\vartheta'_{ab}(0 | \tau)} \right] \;.
\end{equation}
For $m,n \in \mathbb{N}_0$ and $0 \le k \le n$ we have set
\begin{equation}
\label{R_mnk}
 R^{\,ab}_{mnk} = (-1)^{k+ab+1} \; \dfrac{(2n-ab+1)!}{k! \, (n-k)!} \; \dfrac{2^{m-3k-ab-1}}{3^n} \; 
 \dfrac{\vartheta_4 \cdot u^{m+n-k}}{h^{3+2k-ab} \cdot f_{2}} 
\end{equation}
where $u$, $h$, and $f_2$ were defined in (\ref{qforms}).
To evaluate the regularized $u$-plane integral we introduce the non-holomorphic modular form
$Q_{ab}(\tau)=Q^+_{ab}(\tau)+Q^-_{ab}(\tau)$  for $\Gamma_0(4)$ of weight $(3/2-ab)$
such that 
\begin{equation}\label{total_derivative}
 8 \, \sqrt{2} \pi \, i \; \frac{d}{d\bar{\tau}} \; Q_{ab}\left(\tau\right) = y^{-\frac{3}{2}} \, 
 \left[ \overline{\vartheta_{ab}(0 | \tau)} - 4 \, y \,  \overline{\vartheta'_{ab}(0 | \tau)} \right] \;.
\end{equation}
We then have the following extension of \cite[(9.18)]{MooreWitten} which includes the case $a=b=1$:
\begin{lemma}\label{EkQ}
The weakly holomorphic function
\begin{equation}
\label{EkQab}
 \mathcal{E}^k \left[ Q_{ab} \right] = \sum_{j=0}^k (-1)^j \; \binom{k}{j} \; \frac{\Gamma\left(\frac{3}{2}-ab\right)}{\Gamma\left(\frac{3}{2}-ab+j\right)} \; 2^{2j} \; 3^j
 \; E_2^{k-j}(\tau) \; \left(q \, \frac{d}{dq} \right)^j Q_{ab}\left(\tau\right) 
\end{equation}
has weight $2l+3/2-ab$ and satisfies
\begin{equation}
 8\,\sqrt{2} \pi \, i \; \frac{d}{d\bar{\tau}}\; \mathcal{E}^k \left[ Q_{ab} \right]  = 
y^{-\frac{3}{2}} \, \widehat{E}_2^k(\tau) \;
 \left[ \overline{\vartheta_{ab}(0 | \tau)} - 4 \, y \,  \overline{\vartheta'_{ab}(0 | \tau)} \right] \;. 
\end{equation}
\end{lemma}

\subsubsection{The case $a=0$ or $b=0$.} These non-holomorphic modular forms of weight 3/2 for $\Gamma_0(4)$
were constructed by Zagier \cite{Zagier} and reviewed in \cite{MooreWitten}. 
The holomorphic parts of Zagier's weight 3/2 Maass-Eisenstein series,
which first arose \cite{HirzebruchZagier} in connection with
intersection theory for certain Hilbert modular surfaces, are generating functions for Hurwitz class numbers.
The holomorphic part of Zagier's weight 3/2 Maass-Eisenstein series
is the generating function for Hurwitz class numbers.
They have series expansions of the form
\begin{equation}\label{Qplus_0}
 \begin{split}
  Q_{10}^+\left(\tau\right) &= \frac{1}{q^{\frac{1}{8}}} \; \sum_{l > 0} \mathcal{H}_{4l-1} \; q^{\frac{l}{2}} \;,\\
  Q_{00}^+\left(\tau\right) &= \sum_{l \ge 0} \mathcal{H}_{4l} \; q^{\frac{l}{2}} 
 \end{split}
\end{equation}
where $\mathcal{H}_{\alpha}$ are the Hurwitz class numbers. The first non vanishing Hurwitz class numbers are as follows:
\begin{equation*}
\begin{array}{c|c|c|c|c|c|c|c}
 \mathcal{H}_0 & \mathcal{H}_3 & \mathcal{H}_4 & \mathcal{H}_7 & \mathcal{H}_8 & \mathcal{H}_{11} & \mathcal{H}_{12} & \dots \\
 \hline &&&&&&& \\[-2ex]
 - \frac{1}{12} & \frac{1}{3} & \frac{1}{2} & 1 & 1 & 1 & \frac{4}{3} & \dots
\end{array}
\end{equation*}
The non-holomorphic parts have series expansions of the form
\begin{equation}\label{Qminus_0}
\begin{split}
 Q^{-}_{10}\left(\tau\right) & =  \frac{1}{8\sqrt{2\pi}} \; \sum_{l = - \infty}^\infty \; (l+\frac{1}{2}) \cdot \Gamma\left( -\frac{1}{2} , 2 \, \pi \, \left(l+\frac{1}{2}\right)^2 \, y \right) \; q^{-\frac{(l+1/2)^2}{2}} \;,\\
 Q^{-}_{00}\left(\tau\right) & =  \frac{1}{8\sqrt{2\pi}} \; \sum_{l = - \infty}^\infty \; l \cdot \Gamma\left( -\frac{1}{2} , 2 \, \pi \, l^2 \, y \right) \; q^{-\frac{l^2}{2}} \;,
\end{split}
\end{equation}
where $\Gamma(3/2,x)$ is the incomplete gamma function
\begin{equation}
 \Gamma(\alpha,x) = \int_x^\infty e^{-t} \; t^{\alpha-1} \; dt \;.
\end{equation}
We also have set $Q_{01}(\tau)=Q_{00}(\tau)-Q_{10}(\tau)+\frac{1}{2}Q_{00}((\tau+2)/4)$
and write
\begin{equation}
  Q_{01}^+\left(\tau\right) = \sum_{n \ge 0} \mathcal{R}_{n} \; q^{\frac{n}{2}}  \;.
\end{equation}
The first non vanishing coefficients in the series expansion are as follows:
\begin{equation*}
\begin{array}{c|c|c|c|c|c}
 \mathcal{R}_0 & \mathcal{R}_1 & \mathcal{R}_2 & \mathcal{R}_3 & \mathcal{R}_4 & \dots \\
 \hline &&&&& \\[-2ex]
 - \frac{1}{8} & -\frac{1}{4} & \frac{1}{2} & -1 & \frac{5}{4} & \dots
\end{array}
\end{equation*}
All non-holomorphic parts have an exponential decay since
\begin{equation}
\label{decay}
  \Gamma\left( \alpha , t\right) = t^{\alpha-1} \; e^{-t} \; \left( 1 + O(t^{-1}) \right)
\qquad (t \to \infty) \;.
\end{equation}

\subsubsection{The case $a=1$ and $b=1$.}
The harmonic Maass form of weight 1/2 was constructed in \cite{MalmendierOno}.
The holomorphic part has a series expansion of the form
\begin{equation}\label{Qplus}
 Q_{11}^+\left(\tau\right)  = \frac{1}{q^{\frac{1}{8}}} \; \sum_{l \ge 0} H_l \; q^{\frac{l}{2}} 
 = \frac{1}{q^{\frac{1}{8}}}\left( 1 + 28 \, q^{\frac{1}{2}} + 39 \, q + 196 \, q^{\frac{3}{2}} + 161 \, q^2 + \dots \right) 
\end{equation}
where the coefficients $H_l$ were computed in \cite{MalmendierOno}.
The non-holomorphic part $Q_{11}^{-}$ is
\begin{equation}\label{Qminus}
\begin{split}
 Q^{-}_{11}\left(\tau\right) & =  
- \frac{2i}{\sqrt{\pi}} \; \sum_{l \ge 0} (-1)^l \; \Gamma\left( \frac{1}{2} , \, 2\, \pi \, \left(l+\frac{1}{2}\right)^2 \, y \right) \; q^{-\frac{(l+1/2)^2}{2}} \;.
\end{split}
\end{equation}
The non-holomorphic part 
$Q^{-}_{11}$ has an exponential decay similar to the one in Equation (\ref{decay}).
The modular form $Q_{11}(\tau)$ is naturally related to one of Ramanujan's mock theta functions
\begin{equation*}
\begin{split}
M(q^8)& =q^{-1}\sum_{n=0}^{\infty}\frac{(-1)^{n+1}q^{8(n+1)^2}\prod_{k=1}^n
(1-q^{16k-8})}{\prod_{k=1}^{n+1}(1+q^{16k-8})^2}\\
& =-q^7+2q^{15}-3q^{23}+\cdots.
\end{split}
\end{equation*}
In \cite{MalmendierOno} it was proved that $Q_{11}(q^8)+4 \, M(q^8)$
is a weight $1/2$ weakly holomorphic modular form, and
$$
\frac{1}{\eta(\tau)}\cdot \Big( Q_{11}(q)+4 \, M(q)\Big)
$$
is a modular function.

\subsubsection{The evaluation of the $u$-plane integral.}
It was shown in \cite{MooreWitten} that the cusp contribution at $\tau=\infty$ to the 
regularized $u$-plane integral can be evaluated as follows: in Equation (\ref{coefficients}) we integrate by parts using the modular forms constructed in Lemma \ref{EkQ}, i.e., we rewrite an integrand $f$ as a total derivative using
\begin{equation*}
 dx\wedge dy \; \partial_{\bar{\tau}} f  = \frac{1}{2} \, dx \wedge dy \; \left( \partial_{x} + i \, \partial_{y} \right) \, f
= -\frac{i}{2} \, d\Big(  f \, dx + i \, f \, dy \Big) \;.
\end{equation*}
We carry out the integral along the boundary $x=\re{(\tau)} \in [0,4]$ and $y \gg 1$ fixed. This extracts the constant term
coefficient. We then take the limit $y \to \infty$. Since all non-holomorphic parts
have an exponential decay the non-holomorphic dependence drops out. The following expression for
the $u$-plane integral were obtained for the gauge group $\mathrm{SU}(2)$ in \cite{MooreWitten} and
$\mathrm{SO}(3)$ in \cite{MalmendierOno}. Additional information about the evaluation of the $u$-plane integral as well as the geometry of the Seiberg-Witten curve can be found in \cite{M2,M}.
\begin{theorem}
\label{evaluation1}
On $X=\mathbb{C}\mathrm{P}^2$ let $\omega=\mathrm{H}$ be the period point of the metric. 
Let $2\,\lambda_0 = a \, \mathrm{H}$ with $a \in \{0,1\}$ be an integral lifting of $w_2(E)$ and
$b=1$. For $(X,\omega,\lambda_0,w_2=-b \, \mathrm{H})$ 
the $u$-plane integral in the variables $p \, \mathrm{x} \in H_0(X, \mathbb{Z})$,
$S=\kappa \,\mathrm{h} \in H_2(X, \mathbb{Z})$ is
\begin{equation}
 Z=Z_{\tau=\infty} = \sum_{m,n \in \mathbb{N}_0} \dfrac{p^m}{m!} \dfrac{\kappa^{2n-ab+1}}{(2n-ab+1)!} \; D^{\, ab}_{mn} 
\end{equation}
where
\begin{equation}
\label{coefficients}
D^{\, ab}_{mn} =  \sum_{k=0}^n  \Big[  R^{\, ab}_{mnk} \; \; \mathcal{E}^k[Q_{ab}^+(\tau)] \Big]_{q^0}
\end{equation}
and $R^{\,ab}_{mnk}$ and $\mathcal{E}^k[Q_{ab}(\tau)]$ were defined in
(\ref{R_mnk}) and (\ref{EkQab}) respectively.
\end{theorem}

\noindent
For concreteness we list the first non vanishing coefficients of the generating functions 
for $a \in \{0,1\}$ and $b=1$ in Theorem \ref{evaluation1}:
{\footnotesize
\begin{center}
\begin{tabular}{|ll||l|l||l|l|}
\hline
\raisebox{-0.5ex}{$m$} &  \raisebox{-0.5ex}{$n$} & \raisebox{-0.5ex}{$D^{1,1}_{m,n}$} & \raisebox{-0.5ex}{$D^{1,1}_{m,n}$} & \raisebox{-0.5ex}{$D^{0,1}_{m,n}$} & \raisebox{-0.5ex}{$D^{0,1}_{m,n}$} \\ [1ex]
\hline
 \raisebox{-0.5ex}{$0$} &\raisebox{-0.5ex}{$0$}
&\raisebox{-0.5ex}{$1$}&\raisebox{-0.5ex}{$\frac{1}{4} H_1 - 6 H_0$} 
&\raisebox{-0.5ex}{$-\frac{3}{2}$}&\raisebox{-0.5ex}{$-\frac{1}{2} \,  \mathcal{R}_1 + 13 \, \mathcal{R}_0$} 
\\ [1.5ex]
 $0$ & $2$
&$\frac{3}{16}$
&$\frac{49}{64} \, H_2 -  \frac{9}{4} \, H_1 +  \frac{2133}{64} \, H_0$
&$\phantom{-}1$&$-2 \, \mathcal{R}_2 + 7 \, \mathcal{R}_1 -30 \, \mathcal{R}_0$\\ [1ex]
$1$ & $1$
&$\frac{5}{16}$    
&$\frac{7}{64} \, H_2 -  \frac{1}{4}  \, H_1 +   \frac{195}{64}  \, H_0$
&$-1$&$-\frac{1}{4} \, \mathcal{R}_2 + \frac{1}{2} \, \mathcal{R}_1 + 6 \, \mathcal{R}_0$\\[1ex]
 $2$ & $0$
&$\frac{19}{16}  $    
&$\frac{1}{64} \, H_2 +  \frac{1}{4}  \, H_1 -   \frac{411}{64} \,  H_0$
&$-\frac{13}{8}$&$-\frac{1}{32} \, \mathcal{R}_2 - \frac{7}{16} \, \mathcal{R}_1 + \frac{55}{4} \, \mathcal{R}_0$\\ [1ex]
\hline
 \end{tabular}
\end{center}
}
As evidence for their conjecture
in the case of $X=\mathbb{C}\mathrm{P}^2$ and the gauge group $\mathrm{SU}(2)$, Moore and Witten \cite{MooreWitten} computed the first 40 invariants $D^{0,1}_{m,n}$ and found them to be in agreement 
with the results of Kotschick and Lisca \cite{KotschickLisca} and Ellingsrud and G\"ottsche \cite{EllingsrudGoettsche} for the Donaldson invariants.
However, Conjecture \ref{conjecture_light} remains open. 
The main result in \cite{MalmendierOno} concerned the case of the $\mathrm{SO}(3)$-gauge theory.
The following theorem was proved:

\begin{theorem}\label{MainTheorem_MO}
Conjecture \ref{MooreWitten} is true for the gauge group $\mathrm{SO}(3)$, $X=\mathbb{C}\mathrm{P}^2$.
\end{theorem}

\subsection{The blowup of the projective plane}
\label{blowup}
Using the notation of the previous section
we will write down an explicit formula for Donaldson invariants of the connected sum $\widehat{\mathbb{C}\mathrm{P}^2}=\mathbb{C}\mathrm{P}^2 \# \,\overline{\mathbb{C}\mathrm{P}^2}$. 
Let $\widehat{\omega} = \mathrm{H}+=\mathrm{H} - \epsilon \, \mathrm{E}$ with $0< \epsilon \ll 1$ 
be the period point of the metric on the blowup. 
The cycles $\{\mathrm{h}, \mathrm{e}\}$ form a basis of the two-cycles of the blowup. We set $\widehat{S}=\kappa \, \mathrm{h} + \mu \, \mathrm{e}$ such that $\widehat{S}^2=\kappa^2 - \mu^2$. 
We also denote the integral lifting of $w_2(\widehat{\mathbb{C}\mathrm{P}^2})$ by $\widehat{w}_2= b \, (-\mathrm{H} + \mathrm{E})$ with $b=1$.
\begin{lemma}
On $\widehat{X}=\widehat{\mathbb{C}\mathrm{P}^2}$ let $\widehat{\omega}= \mathrm{H}+$
be the period point of the metric. Let $2\,\widehat{\lambda}_0=a \, (\mathrm{H} - \mathrm{E})$ with $a \in \{0,1\}$ be an integral lifting of $w_2(\widehat{E})$ and $b=1$. 
For $(\widehat{X},\widehat{\omega},\widehat{\lambda}_0,\widehat{w}_2=b \, (-\mathrm{H} + \mathrm{E}))$
the Siegel-Narain theta function is
\begin{equation}
\label{siegnar_CP2_bu}
\begin{split}
\widehat{\bar \Theta}_{ab}  = (-1)^{ab} \; \exp{\left(\frac{\pi}{2\, y} \, \bar{\xi}_+^2+\frac{1}{8\pi \, y \, h^2} \, \mu^2\right)} \; \overline{\vartheta_{ab}\Big((\xi_+,\mathrm{H}) \Big| \tau\Big)}\; \vartheta_{ab}\left(\left.\frac{\mu}{2\pi h}\right|\tau\right) 
\end{split}
\end{equation}
where $h$ was defined in (\ref{qforms}).
\end{lemma}
\begin{proof}
Evaluating the Siegel-Narain theta function (\ref{siegnar}) for $\widehat{\mathbb{C}\mathrm{P}^2}$
gives an additional factor of
\begin{equation}
 \vartheta_{a(-b)}\left(\left.\frac{\mu}{2\pi h}\right|\tau\right) =
 \vartheta_{ab}\left(\left.-\frac{\mu}{2\pi h}\right|\tau\right) =
 (-1)^{ab} \, \vartheta_{ab}\left(\left.\frac{\mu}{2\pi h}\right|\tau\right) 
 \;.
\end{equation}
Finally, the result follows from $\vartheta_{(-a)b}(v|\tau)=\vartheta_{ab}(v|\tau)$.
\end{proof}
\noindent
The following is a reformulation of Theorems \ref{blowupSO3}, \ref{blowupSU2}, and \ref{evaluation2}
applied to the blowup of the projective plane:
\begin{corollary}
\label{evaluation2}
On $\widehat{X}= \widehat{\mathbb{C}\mathrm{P}^2}$
let $\widehat{\omega}= \mathrm{H}+$ be the period point of the metric. Let
$2\,\widehat{\lambda}_0=a \, (\mathrm{H} -  \mathrm{E})$ with $a \in \{0,1\}$ be an integral lifting of $w_2(\widehat{E})$ and $b=1$. For $(\widehat{X},\widehat{\omega},\widehat{\lambda}_0,\widehat{w}_2=b \, (-\mathrm{H} + \mathrm{E}))$ the $u$-plane integral in the variables $p \, \mathrm{x} \in H_0(\widehat{X}, \mathbb{Z})$ and $\widehat{S}=\kappa \,\mathrm{h} + \mu \, \mathrm{e}\in H_2(\widehat{X}, \mathbb{Z})$ is
\begin{equation}
\label{generating_function}
 Z=Z_{\tau=\infty}
  = \sum_{m,n \in \mathbb{N}_0} \dfrac{p^m}{m!} \dfrac{\kappa^{2n-ab+1}}{(2n-ab+1)!} \; \widehat{D}^{\,ab}_{mn} \;
\end{equation}
where
\begin{equation}
\label{coefficients2}
\widehat{D}^{\,ab}_{mn} =  \sum_{k=0}^n  \left[ R^{\,ab}_{mnk} \; 
e^{-\mu^2 \, T} \; \frac{\vartheta_{ab}\left(\frac{\mu}{2\pi h} \Big| \tau \right)}{\vartheta_4(\tau)}
\; \mathcal{E}^k[Q_{ab}^+(\tau)] \right]_{q^0}
\end{equation}
and $T$, $R^{\,ab}_{mnk}$, and $\mathcal{E}^k[Q_{ab}(\tau)]$ were defined in (\ref{qforms}),
(\ref{R_mnk}), and (\ref{EkQab}) respectively.
\end{corollary}

\noindent 
The $u$-plane integral agrees with the Donaldson invariants and also satisfies
the same blowup formulas (\ref{blowup_SO3}) and (\ref{blowup_SU2}) and the same
wall-crossing formula (\ref{wcross}). Thus, it follows immediately:
\begin{corollary}\label{MW_CP2bu}
Conjecture \ref{MooreWitten} is true for the gauge group $\mathrm{SO}(3)$, $X=\widehat{\mathbb{C}\mathrm{P}^2}$.
\end{corollary}

\begin{remark}
We remark that Equations (\ref{generating_function}) and (\ref{coefficients}) also make sense for $b=0$.
For $b=0$, they are the evaluation of the $u$-plane integral for the Siegel-Narain theta function
for the non-geometric lattice
$H^2 = \lbrace N_1 \, \mathrm{H} + N_2 \, \mathrm{E} \, |\,  N_1, N_2 \in \mathbb{N}_0 \rbrace $ together with
$2 \, \lambda_0 = a \, (\mathrm{H} - \mathrm{E})$, $w_2 = b \, (-\mathrm{H} + \mathrm{E})$, 
$\omega = \mathrm{H} - \epsilon \mathrm{E}$.
\end{remark}

\subsection{$\mathbb{C}\mathrm{P}^1 \times \mathbb{C}\mathrm{P}^1$ in limiting chambers}
\label{unfolding_technique}

On $X=\mathbb{C}\mathrm{P}^1 \times \mathbb{C}\mathrm{P}^1$ we will use 
the formal variable $S=2\,\kappa_f \, \mathrm{f} + \kappa_g \, \mathrm{g}$ with $S^2=4\, \kappa_f\kappa_g$ in the $u$-plane integral. $X$ has $\sigma=0$ and is spin, hence $w_2(X)=0$. We choose the period point
\begin{equation}
 \omega = \frac{1}{\sqrt{2} \, \epsilon} \, \mathrm{F} + \frac{\epsilon}{\sqrt{2}} \, \mathrm{G} \;
 \text{such that} \; \omega^2 = 1\;. 
\end{equation}
First, we will describe the Donaldson invariants in the limiting chamber $\epsilon \to 0$
which corresponds to a small volume of $\mathrm{F}$ since $\int_\mathrm{f} \omega = 
\sqrt{2} \, \epsilon$. As explained in \cite{MooreWitten} the metric 
has positive scalar curvature in this chamber and hence
\begin{equation}
Z_{\tau=0} = Z_{\tau=2} =0 \;.
\end{equation}
Thus, the evaluation of the $u$-plane integral in this chamber amounts to evaluating the cusp contribution
at $\tau=\infty$. As explained in \cite{MooreWitten} the $u$-plane integral can be evaluated by a general
strategy due to Borcherds, called lattice reduction method or unfolding technique. The evaluation was described in great detail in \cite{LozanoMarino}. The unfolding technique yields the following results:
\begin{theorem}
\label{evaluation3}
On $X=\mathbb{C}\mathrm{P}^1 \times \mathbb{C}\mathrm{P}^1$
let $\omega= \mathrm{F}+$ 
be the period point of the metric. Let
$2\,\lambda_0=\rho_f \, \mathrm{F} + \rho_g \, \mathrm{G}$ with $\rho_f, \rho_g 
\in \{0,1\}$ be an integral lifting of $w_2(E)$. 
For $(X,\omega,\lambda_0,w_2=0)$ 
the $u$-plane integral in the variables $p \, \mathrm{x} \in H_0(X, \mathbb{Z})$ and
$S=\kappa_f \, \mathrm{f} + 2 \, \kappa_g \, \mathrm{g} \in H_2(X, \mathbb{Z})$ is 
\begin{equation}
\label{unfolding}
 Z=Z_{\tau=\infty}= \left\lbrace \begin{array}{ll} 
 - \frac{1}{4} \left\lbrack \frac{1}{h \,\cdot f_2} \; e^{2 \, p \, u + S^2 \,T} \;
 \cot\left( \frac{\kappa_g}{h}\right)\right\rbrack_{q^0} & \text{if $\rho_g =0, \rho_f=0$}\\
  - \frac{1}{4} \left\lbrack \frac{1}{h \,\cdot f_2} \; e^{2 \, p \, u + S^2 \,T} \;
 \csc\left( \frac{\kappa_g}{h}\right)\right\rbrack_{q^0} & \text{if $\rho_g =0, \rho_f=1$}\\[2ex]
 \phantom{-}0 & \text{if $\rho_g\not = 0$}
\end{array} \right. 
\end{equation}
where $h, f_2, T, u$ were defined in (\ref{qforms}).
\end{theorem}

\begin{corollary}
\label{MW_CP1xCP1}
Conjecture \ref{MooreWitten} is true for $X=\mathbb{C}\mathrm{P}^1 \times \mathbb{C}\mathrm{P}^1$.
\end{corollary}
\begin{proof}
The $u$-plane integrals (\ref{unfolding}) of $\mathbb{C}\mathrm{P}^1 \times \mathbb{C}\mathrm{P}^1$ in the 
limiting chamber $\omega=\mathrm{F}+$ agree precisely with the formulas computed by G\"ottsche for the 
Donaldson invariants in \cite{Goettsche}. Moreover, Moore and Witten showed that the $u$-plane integral satisfies the wall-crossing formula in Equation (\ref{wcross}) which
agrees with the wall-crossing formula for the Donaldson invariants derived in \cite{Goettsche} under the assumption of the Kotschick-Morgan conjecture. The assumption of the Kotschick-Morgan conjecture was later
removed in \cite{GoettscheNakajimaYoshioka}.
\end{proof}

If we choose the period point $\omega =   \frac{1}{2} \, \mathrm{F} +  \mathrm{G}$ on
$\mathbb{C}\mathrm{P}^1 \times \mathbb{C}\mathrm{P}^1$ then the unfolding technique 
used to obtain the results for the $u$-plane integral stated in Chapter \ref{unfolding_technique} 
can no longer be applied. In principle, one could still
use the wall-crossing formula (\ref{wcross}) to relate the results from Theorem \ref{evaluation3} for the
period point $\omega = \mathrm{F}+$ to the $u$-plane integral for the period point $\omega = \frac{1}{2} \,   \mathrm{F} +  \mathrm{G}$. However, since an infinite
number of walls needs to be crossed in the process it is difficult to obtain an explicit formula in this way. Instead we will use a different approach. The insertion of a $\mathbb{Z}_2$-delta function
into the Siegel-Narain theta function (\ref{siegnar}) on $\mathbb{C}\mathrm{P}^1 \times \mathbb{C}\mathrm{P}^1$ allows an evaluation of the $u$-plane integral using the method of Chapter \ref{blowup}. We will show:
\begin{proposition}
\label{prop1}
On $X=\mathbb{C}\mathrm{P}^1 \times \mathbb{C}\mathrm{P}^1$
let $\omega=  \frac{1}{2} \, \mathrm{F} +  \mathrm{G}$ be the period point of the metric.
For $(X,\omega,\lambda_0=0,w_2=0)$ the $u$-plane integral in the variables 
$p \, \mathrm{x} \in H_0(X, \mathbb{Z})$, $S=\kappa_f \, \mathrm{f} + 2 \, \kappa_g \, \mathrm{g} \in H_2(X, \mathbb{Z})$
is
\begin{equation}
\label{generating_function2}
 Z_{\tau=\infty}
  =  \frac{1}{2} \, \sum_{a,b \,\in \lbrace 0,1\rbrace}
  \; \sum_{m,n \in \mathbb{N}_0} \dfrac{p^m}{m!} \;\dfrac{\kappa^{2n-ab+1}}{(2n-ab+1)!} \; (-1)^{ab} \; \widehat{D}^{\,ab}_{mn} \;
\end{equation}
where $\widehat{D}^{\,ab}_{mn}$ was defined in (\ref{coefficients2}), and
$\kappa= \kappa_f + \kappa_g$, $\mu=-\kappa_f + \kappa_g$.
\end{proposition}
\noindent
We will give the detailed proof in Section \ref{proof}.

\subsection{Relations between $\mathbb{C}\mathrm{P}^1 \times \mathbb{C}\mathrm{P}^1$ and $\widehat{\mathbb{C}\mathrm{P}^2}$}
To state certain relations between the Donaldson invariants of $\mathbb{C}\mathrm{P}^1 \times \mathbb{C}\mathrm{P}^1$ and the Donaldson invariants of $\widehat{\mathbb{C}\mathrm{P}^2}$ 
we introduce  $\bar{\mathrm{F}}= \mathrm{H} - \mathrm{E}$ 
and $\bar{\mathrm{G}}= \mathrm{H} + \mathrm{E}$
on $\widehat{\mathbb{C}\mathrm{P}^2}$ such that $(\bar{\mathrm{F}},\bar{\mathrm{G}})=2$. 
On $\mathbb{C}\mathrm{P}^1 \times \mathbb{C}\mathrm{P}^1$ we set
$S = \kappa_f \, \mathrm{f} +  2\, \kappa_g \, \mathrm{g}$, and
on $\widehat{\mathbb{C}\mathrm{P}^2}$ we set
$\widehat{S} = \kappa_f \, \bar{\mathrm{f}} + \kappa_g \, \bar{\mathrm{g}}$.
We have the following result:
\begin{lemma}
On $\mathbb{C}\mathrm{P}^1 \times \mathbb{C}\mathrm{P}^1$
let $\omega= \frac{1}{2} \,  \mathrm{F} +\mathrm{G}$ be the period point of the metric.
Under the assumption of Conjecture \ref{conjecture_light} it follows that
\begin{equation}
\label{relations}
\begin{split}
Z_{\tau=\infty} \Big( \mathbb{C}\mathrm{P}^1 \times \mathbb{C}\mathrm{P}^1, \omega, \frac{1}{2} \, \mathrm{F}, 
 0 \Big| \, p, S \Big) 
&- Z_{\tau=\infty} \Big( \mathbb{C}\mathrm{P}^1 \times \mathbb{C}\mathrm{P}^1, \omega, 0,
 0 \Big| \, p, S \Big) \\
= \quad  Z_{\tau=\infty} \Big( \widehat{\mathbb{C}\mathrm{P}^2}, \mathrm{H}+, \frac{1}{2} \, \bar{\mathrm{F}}, 
   - \bar{\mathrm{F}} \Big| \, p, \widehat{S} \Big)
&- Z_{\tau=\infty} \Big( \widehat{\mathbb{C}\mathrm{P}^2},  \mathrm{H}+,  0,
   - \bar{\mathrm{F}} \Big| \, p, \widehat{S} \Big) \quad.
\end{split}
\end{equation}
\end{lemma}

\begin{remark}
In \cite[Theorem 5.3]{GoettscheZagier}, G\"ottsche and Zagier proved a general version of the relation (\ref{relations}) for the generating function of Donaldson invariants of $\mathbb{C}\mathrm{P}^1 \times \mathbb{C}\mathrm{P}^1$ and the blowup of $\mathbb{C}\mathrm{P}^2$. The relation (\ref{relations}) but for the generating
function of the Donaldson invariants is obtained by setting
$a=(1+\epsilon)/2$ and $b=(1-\epsilon)/2)$ in \cite[Theorem 5.3]{GoettscheZagier} and then taking the limit $\epsilon \to 0$.
\end{remark}

\begin{proof}
Based on the previous remark the proof follows from the fact that the $u$-plane integral is equal to the generating function for the Donaldson invariants. For $X=\mathbb{C}\mathrm{P}^1 \times \mathbb{C}\mathrm{P}^1$ this follows from Corollary \ref{MW_CP1xCP1}. For the gauge group 
$\mathrm{SO}(3)$ and $\mathbb{C}\mathrm{P}^2\# \,\overline{\mathbb{C}\mathrm{P}^2}$ the statement follows from Corollary \ref{MW_CP2bu}. For the gauge group $\mathrm{SU}(2)$ we need to assume that Conjecture \ref{conjecture_light} is true. Under this assumption it follows that the $u$-plane integral for the gauge group $\mathrm{SU}(2)$ on $\mathbb{C}\mathrm{P}^2\# \,\overline{\mathbb{C}\mathrm{P}^2}$ agrees with the $\mathrm{SU}(2)$-Donaldson invariants since both satisfy the same blowup formula (\ref{blowup_SU2})
and the same wall-crossing formula (\ref{wcross}).
\end{proof}

\begin{corollary}
On $\mathbb{C}\mathrm{P}^1 \times \mathbb{C}\mathrm{P}^1$
let $\omega=  \frac{1}{2} \, \mathrm{F} + \mathrm{G}$ be the period point of the metric. Let
$2\,\lambda_0= \mathrm{F}$ be an integral lifting of $w_2(E)$. 
Under the assumption of Conjecture \ref{conjecture_light} it follows that
for $(X,\omega,\lambda_0,w_2=0)$ the $u$-plane integral in the variables
$p \, \mathrm{x} \in H_0(X, \mathbb{Z})$, $S=\kappa_f \, \mathrm{f} + 2 \, \kappa_g \, \mathrm{g} \in H_2(X, \mathbb{Z})$
is
\begin{equation}
\label{formula2}
 Z_{\tau=\infty}
  = \frac{1}{2} \, \sum_{a,b\,\in \lbrace 0,1\rbrace}
  \;  \sum_{m,n \in \mathbb{N}_0} \dfrac{p^m}{m!} \;  \dfrac{\kappa^{2n-ab+1}}{(2n-ab+1)!} \; (-1)^{(a+1)b} \; \widehat{D}^{\,ab}_{mn} 
\end{equation}
where $\widehat{D}^{\,ab}_{mn}$ was defined in (\ref{coefficients2}),
and $\kappa= \kappa_f + \kappa_g$, $\mu=-\kappa_f + \kappa_g$.
\end{corollary}
\begin{proof}
Assuming Conjecture \ref{conjecture_light}, Equation (\ref{relations}) 
provides a formula for the $u$-plane integral for the gauge group $\mathrm{SU}(2)$
on $\mathbb{C}\mathrm{P}^1 \times \mathbb{C}\mathrm{P}^1$ in terms of the
$u$-plane integral for the gauge group $\mathrm{SO}(3)$ on $\mathbb{C}\mathrm{P}^1 \times \mathbb{C}\mathrm{P}^1$ and the $u$-plane integrals on $\widehat{\mathbb{C}\mathrm{P}^2}$.
Using the results of Proposition \ref{prop1} and Corollary \ref{evaluation2} we obtain Equation
(\ref{formula2}).
\end{proof}

\section{The proof of Theorem \ref{main_prop}}
\label{proof}
\begin{proof}
We first prove Proposition \ref{prop1}.
To do so we compute the Siegel Narain theta function of $\mathbb{C}\mathrm{P}^1 \times \mathbb{C}\mathrm{P}^1$ for $w_2=0$, $\lambda_0=0$, and the period point 
\begin{equation}
 \omega = \left( \frac{1}{2} + \frac{\epsilon}{2} \right) \, \mathrm{F} + ( 1 - \epsilon ) \, \mathrm{G} \;.
\end{equation}
If we write $\lambda = M_1 \, \mathrm{F} + M_2 \, \mathrm{G}$ with $M_1,M_2 \in \mathbb{Z}$
we obtain
\begin{equation}
 \begin{split}
  -i \pi  \bar{\tau} \, (\lambda_+)^2 & =  -i \pi  \bar{\tau} \; \dfrac{ (\lambda,\omega)^2}{(\omega,\omega)} 
=  -i \pi  \bar{\tau} \, \dfrac{ \left\lbrack 2 \, M_1 \, (1-\epsilon) +  M_2 \, (1+\epsilon) \right\rbrack^2}{4 \, (1-\epsilon^2)} \;,\\
- i \pi  \tau \, (\lambda_-)^2 & = -i \pi \tau \, (\lambda - \lambda_+)^2 
=  i \pi  \tau \; \dfrac{ \left\lbrack 2\, M_1 \, (1-\epsilon) - M_2 \, (1+\epsilon) \right\rbrack^2}{4 \, (1-\epsilon^2)} \;.
 \end{split}
\end{equation}
For $\bar{\xi} = \rho \, y \, h \,\omega + S_-/(2\pi h)$  and $S= \kappa_f \, \mathrm{f} + 2\,\kappa_g \, \mathrm{g}$ 
we also have that 
\begin{equation}
 \begin{split}
 (\xi_+)^2 & = \rho^2 \, y^2 \, h^2 \, (1- \epsilon^2) \;,\\
 (\bar{\xi}_-)^2 & = - \dfrac{\left\lbrack \kappa_f (1-\epsilon) - \kappa_g \, (1+\epsilon) \right\rbrack^2}{
 4 \, \pi^2 \, h^2 (1-\epsilon^2)} \;,\\
 \frac{\pi}{2y} \Big(\bar{\xi}_{+}^2-\bar{\xi}_{-}^2\Big) & =
 \frac{1}{2} \, \pi \, y \, \rho^2 \, h^2 \, (1-\epsilon^2)
  - \dfrac{ \left\lbrack  \kappa_f \, (1-\epsilon) - \kappa_g \, (1+\epsilon) \right\rbrack}{8 \, \pi \, y \, h^2 \, (1-\epsilon^2)}\;,\\
- 2 \pi i\, (\lambda,\bar\xi) & = - i \pi\, y \, h \, \rho \left\lbrack 2\, M_1 \, (1-\epsilon) +  M_2\, (1+\epsilon)\right\rbrack\\
& + \dfrac{i \, \left\lbrack 2\, M_1 \, (1-\epsilon) -  M_2 \, (1+\epsilon)\right\rbrack 
\left\lbrack \kappa_f \, (1- \epsilon) - \kappa_g \, (1+\epsilon)\right\rbrack}{2 \, h \, (1-\epsilon^2)} \;.
 \end{split}
 \end{equation}
In the equation for the Siegel-Narain theta function 
\begin{equation}
\label{SNTf}
\begin{split}
\bar \Theta(\xi)\,   = \, &\exp\left\lbrack \frac{\pi}{2y} \Big(\bar{\xi}_{+}^2-\bar{\xi}_{-}^2\Big)\right\rbrack \\
\times  \sum_{M_1,M_2 \in \mathbb{Z} }
& \exp\Big\lbrack - i \pi \bar\tau (\lambda_+)^2 - i \pi   \tau(\lambda_-  )^2
- 2 \pi i\, (\lambda,\bar\xi) \Big\rbrack 
\end{split}
\end{equation}
we replace $M_1 = (N_1 - N_2)/2$ and $M_1= N_1 + N_2 + a$,
insert the $\mathbb{Z}_2$-delta function
\begin{equation}
 \delta\Big(N_1 -N_2 \; (\mathrm{mod} \, 2) \Big) = \frac{1}{2} \, \sum_{b \, \in \lbrace 0,1\rbrace} \;
 \exp{\left\lbrack \pi i \, (N_1-N_2) \, b \,\right\rbrack} \;,
\end{equation}
and take the sums over $N_1, N_2 \in \mathbb{Z}$ and $a,b \in \lbrace 0,1\rbrace$. We obtain
\begin{equation}
\label{SNTf2}
\begin{split}
\bar \Theta(\xi)\,   = \, & \frac{1}{2} \, \sum_{a,b \, \in \lbrace 0,1\rbrace} \; \exp\left\lbrack \frac{\pi}{2y} \Big(\bar{\xi}_{+}^2-\bar{\xi}_{-}^2\Big)\right\rbrack \\
\times  \sum_{N_1,N_2 \in \mathbb{Z} }
& \exp\Big\lbrack - i \pi \bar\tau (\lambda_+)^2 - i \pi   \tau(\lambda_-  )^2
- 2 \pi i\, (\lambda,\bar\xi) + \pi i \, (N_1-N_2) \, b \,\Big\rbrack  \;.
\end{split}
\end{equation}
A short calculation shows that (\ref{SNTf2}) is equal to 
\begin{equation}
\bar{\Theta}(\xi) \, = \frac{1}{2} \, \sum_{a,b \in \lbrace 0,1\rbrace} \, (-1)^{ab} \; \widehat{\bar{\Theta}}_{ab}(\xi) 
\end{equation}
where $\widehat{\bar{\Theta}}_{ab}(\xi)$ is the Siegel-Narain theta function for the lattice
$H^2 = \lbrace N_1 \, \mathrm{H} + N_2 \, \mathrm{E} \, |\,  N_1, N_2 \in \mathbb{Z} \rbrace $ together with
$2 \, \widehat{\lambda}_0 = a \, (\mathrm{H} - \mathrm{E})$, $\widehat{w}_2 = b \, (-\mathrm{H} + \mathrm{E})$, 
and $\widehat{\omega} = \mathrm{H} - \epsilon\, \mathrm{E}$. If we also set $\widehat{S}=
\kappa \, \mathrm{h} + \mu \, \mathrm{e}$ with $\kappa=\kappa_f + \kappa_g$ and $\mu=-\kappa_f+\kappa_g$
then $\widehat{\bar{\Theta}}_{ab}(\xi)$ is the Siegel-Narain theta
function already evaluated in Equation (\ref{siegnar_CP2_bu}).
Since $S^2=\widehat{S}^2$ and $(S,\omega)=(\widehat{S},\widehat{\omega})$ applying Corollary \ref{evaluation2} and taking the limit $\epsilon \to 0$ then yields (\ref{generating_function2}).
This concludes the proof of Proposition \ref{prop1}. 

\bigskip

For $\mathbb{C}\mathrm{P}^2 \# \,\overline{\mathbb{C}\mathrm{P}^2}$ the $u$-plane integral was evaluated in (\ref{generating_function}). Corollary~\ref{MW_CP2bu} proves that the $u$-plane integral equals the generating 
function for the Donaldson invariants for the gauge group $\mathrm{SO}(3)$. This proves Part (1) of Proposition 
\ref{main_prop}. The same statement is true for the gauge group  $\mathrm{SU}(2)$ if we assume Conjecture \ref{conjecture_light}.  The reason is that the $u$-plane integral already satisfies the same blowup formula (\ref{blowup_SU2}) and the same wall-crossing formula (\ref{wcross}) as the generating function for the Donaldson invariants.
This proves Part (2) of Theorem \ref{main_prop}.

Corollary \ref{MW_CP1xCP1} proves that the $u$-plane integrals (\ref{generating_function2}) and (\ref{formula2}) are the generating functions for the Donaldson invariants of $\mathbb{C}\mathrm{P}^1 \times \mathbb{C}\mathrm{P}^1$ for the gauge groups $\mathrm{SU}(2)$ and $\mathrm{SO}(3)$ respectively. Equations (\ref{generating_function2}) and (\ref{formula2}) provide the explicit formulas for the $u$-plane integrals. In the case of the gauge group $\mathrm{SU}(2)$ the assumption
of Conjecture \ref{conjecture_light} is needed since we have used Equation~(\ref{relations}).
This proves Part (3) and (4) of Theorem \ref{main_prop}.
\end{proof}

\section*{Acknowledgments}
\noindent
I would like to thank David Morrison and Ken Ono for many helpful discussions.


\begin{thebibliography}{GKZ}

 \bibitem{AbramowitzStegun}
\emph{M. Abramowitz, I. Stegun},
Handbook of mathematical functions with formulas, graphs, and mathematical tables.
National Bureau of Standards Applied Mathematics Series, 55 For sale by the Superintendent of Documents, U.S. Government Printing Office, Washington, D.C. 1964. 

\bibitem{EllingsrudGoettsche}
\emph{G. Ellingsrud, L. G\"ottsche}, 
Wall-crossing formulas, the Bott residue formula and the Donaldson invariants of rational surfaces,  
Quart. J. Math. Oxford Ser.(2) \textbf{49} (1998), no. 195, pages 307-329. 

\bibitem{FintushelStern}
\emph{R. Fintushel, R. Stern},
The blowup formula for Donaldson invariants,
Ann. of Math. (2) \textbf{143} (1996), no. 3, pages 529-546. 

\bibitem{Goettsche}
\emph{L. G\"ottsche},
Modular forms and Donaldson invariants for $4$-manifolds with $b\sb {+}=1$,
J. Amer. Math. Soc. \textbf{9} (1996), no. 3, pages 827-843. 

\bibitem{Goettsche2}
\emph{L. G\"ottsche},
Donaldson invariants in Algebraic Geometry,
School on Algebraic Geometry (Trieste, 1999),  pages 101-134,
ICTP Lect. Notes, 1, Abdus Salam Int. Cent. Theoret. Phys., Trieste, 2000. 

\bibitem{GoettscheZagier}
\emph{L. G\"ottsche, D. Zagier}, 
Jacobi forms and the structure of Donaldson invariants for $4$-manifolds with $b\sb +=1$,
Selecta Math. \textbf{4} (1998), no. 1, pages 69-115. 

\bibitem{GoettscheNakajimaYoshioka}
\emph{L. G\"ottsche, H. Nakajima,  K. Yoshioka}, 
Instanton counting and Donaldson invariants,
J. Differential Geom.  \textbf{80} (2008), no. 3, pages 343--390.

\bibitem{HirzebruchZagier}  
\emph{F. Hirzebruch, D. Zagier,} 
Intersection numbers of curves on Hilbert modular surfaces and modular forms of Nebentypus,
Invent. Math. \textbf{36} (1976), pages 57-113.

\bibitem{Kotschick}
\emph{D. Kotschick},
$\mathrm{SO}(3)$-invariants for 4-manifolds with $b^+=1$,
Proc. London Math. Soc. (3) \textbf{63} (1991), no.~2, pages 426-448.

\bibitem{KotschickLisca}
\emph{D. Kotschick, P. Lisca},
Instanton invariants of $\mathbb{C}\mathrm{P}^2$ via topology,
Math. Ann. \textbf{303} (1995), no. 2, pages 345-371.

\bibitem{KotschickMorgan}
\emph{D.~Kotschick, J.~Morgan},
$\mathrm{SO}(3)$-invariants for 4-manifolds with $b^+=1$ II,
J. Differential Geom. \textbf{39} (1994), pages 433-456.

\bibitem{KronheimerMrowka}
\emph{P. Kronheimer, T. Mrowka}, 
The genus of embedded surfaces in the projective plane, 
Math. Res. Lett. \textbf{1} (1994), 797-808.

\bibitem{Leness}
\emph{T. Leness},
Blow-up formulae for ${\rm SO}(3)$-Donaldson polynomials,
Math. Z. \textbf{227} (1998), no. 1, pages 1-26. 

\bibitem{LabastidaLozano}
\emph{J. Labastida, C. Lozano},
Duality in twisted $\mathcal{N}=4$ supersymmetric gauge theories in four dimensions,
Nuclear Phys. B  \textbf{537} (1999),  no. 1-3, pages 203--242. 

\bibitem{LozanoMarino}
\emph{C. Lozano, M. Mari\~no},
Donaldson invariants on product ruled surfaces and two-dimensional gauge theories,
Commun. Math. Phys. \textbf{220} (2001), pages 231-261.

\bibitem{MarinoMoore}
\emph{M. Mari\~no, G. Moore},
Integrating over the Coulomb branch in $\mathcal{N}=2$ gauge theory,
Strings '97 (Amsterdam, 1997).
Nuclear Phys. B Proc. Suppl. \textbf{68} (1998), pages 336-347. 

\bibitem{M2} 
\emph{A. Malmendier}, 
Expressions for the generating function of the Donaldson invariants for $\mathbb{C}\mathrm{P}^2$,
Ph.D. Thesis, MIT, 2007.

\bibitem{M}
\emph{A. Malmendier},
The signature of the Seiberg-Witten surface,
Surveys in differential geometry. Vol. XV. Perspectives in mathematics and physics: Essays dedicated to Isadore Singer's 85 birthday. Edited by Tomasz Mrowka and Shing-Tung Yau.
Surveys in Differential Geometry, 15. International Press, Somerville, MA, 2011, in press. {\tt arXiv:0802.1363 [math.DG]}.

\bibitem{MalmendierOno}
\emph{A. Malmendier, K. Ono},
$\mathrm{SO}(3)$-Donaldson invariants of $\mathbb{C}\mathrm{P}^2$ and mock Theta Functions,
{\tt arXiv:0808.1442 [math.DG]}.

\bibitem{MooreWitten}
\emph{G. Moore, E. Witten},
Integration over the $u$-plane in Donaldson theory,  
Adv. Theor. Math. Phys.  \textbf{1}  (1997),  no. 2, pages 298-387.

\bibitem{Ono}
\emph{K. Ono},
The Last Words of a Genius,
Notices Amer. Math. Soc. \textbf{57} (2010), no. 11, 1410-1419. 

\bibitem{SeibergWitten1}
\emph{N. Seiberg, E. Witten},
Electric-magnetic duality, monopole condensation, and confinement in $N=2$ supersymmetric Yang-Mills theory,  Nuclear Phys. B  \textbf{426}  (1994),  no. 1, pages 19-52.

\bibitem{SeibergWitten2}
\emph{N. Seiberg, E. Witten}, 
Monopoles, duality and chiral symmetry breaking in $N=2$ supersymmetric QCD,  
Nuclear Phys. B  \textbf{431}  (1994),  no. 3, pages 484-550. 

\bibitem{Witten1}
\emph{E. Witten},
On $S$-duality in Abelian gauge theory,
Selecta Math. \textbf{1} (1995), no. 2, pages 383-410. 

\bibitem{Witten2}
\emph{E. Witten}
Topological quantum field theory,  Comm. Math. Phys.  \textbf{117}  (1988),  no. 3, pages 353-386.

\bibitem{Witten3}
\emph{E. Witten}
Monopoles and four-manifolds,  Math. Res. Lett.  \textbf{1}  (1994),  no. 6, pages 769-796. 

\bibitem{Zagier} \emph{D. Zagier}, Nombres de classes et formes modulaires de poids 3/2,
C. R. Acad. Sci. Paris S\'er. \textbf{281A}  (1975), no. 21, pages 883-A886.

\bibitem{ZagierBourbaki} \emph{D. Zagier}, Ramanujan's mock theta functions
and their applications [d'apr\`es Zwegers and Bringmann-Ono], S\'em. Bourbaki,
Vol. 2007/2008. Ast\'erisque No. 326 (2009), Exp. No. 986, viiÐviii, 143Ð164 (2010).

\bibitem{Z2} \emph{S. P. Zwegers}, {Mock theta functions},
Ph.D. Thesis, Universiteit Utrecht, 2002.

\end{thebibliography}
\end{document}